\documentclass{article}
\usepackage{graphicx}
\usepackage{amssymb}
\usepackage{amsmath}
\usepackage{a4wide}
\usepackage{caption}
\usepackage{subcaption}
\usepackage[affil-it]{authblk}
\usepackage{xr}
\usepackage{color}
\usepackage{soul}
\usepackage{mathtools}

\usepackage[utf8]{inputenc} 
\usepackage[T1]{fontenc}
\newtheorem{theorem}{Theorem}[section]

\newtheorem{corollary}[theorem]{Corollary}

\newtheorem{lemma}[theorem]{Lemma}

\newtheorem{proposition}[theorem]{Proposition}
\newtheorem{remark}[theorem]{Remark}

\newenvironment{proof}[1][Proof]{\textbf{#1.} }{\ \rule{0.5em}{0.5em}}

\begin{document}
\title{On generalization of Zermelo navigation problem \\ on Riemannian manifolds}
\author{Piotr Kopacz}
\affil{\small{Jagiellonian University, Faculty of Mathematics and Computer Science\\ 6, Prof. S. Łojasiewicza, 30 - 348, Kraków,
Poland}}
\affil{Gdynia Maritime University, Faculty of Navigation \\3,  Al. Jana Pawła II, 81-345, Gdynia, Poland}
\date{\texttt{}}
\maketitle
\begin{abstract}
\noindent
We generalize the Zermelo navigation problem and its solution on Riemannian manifolds $(M, h)$ admitting a space dependence of a ship's own speed $|u(x)|_h\leq1$ in the presence of a perturbation $W$ determined by a mild velocity vector field $|W(x)|_h<|u(x)|_h$, with application of Finsler metric of Randers type. 
\end{abstract}

\ 

\smallskip
\noindent
\textbf{M.S.C. 2010}: 53B20, 53C21, 53C22, 53C60, 49J15, 49J53.



\section{Introduction}

The Zermelo navigation problem is a time-optimal control problem where the objective is to find the minimum time trajectories of a ship sailing in a space $M$, under the influence of a perturbation represented by a vector field $W$ thought of as a wind or a current. The problem was solved by E. Zermelo himself (1931) in $\mathbb{R}^2$ and $\mathbb{R}^3$ \cite{zermelo2, zermelo}, and generalized considerably (2004) in \cite{colleen_shen} for the case when sea a Riemannian manifold $(M, h)$ under the assumption that a wind $W$ is a time-independent weak wind, i.e. $h(W, W)<1$. In the absence of a wind the solutions to the problem are simply $h$-geodesics of $M$. The problem may be investigated as purely geometric. It has been found out that the path minimizing travel-time are exactly the geodesics of a special Finsler type $F$, namely Randers metric \cite{colleen_shen, chern_shen}. In other words solutions to the problem are the flows of Randers geodesics. The strong convexity condition $|W|_h<1$ ensures that $F$ is a positive definite Finsler metric. Furthermore, there is an equivalence between Randers metrics and Zermelo's problems \cite{colleen_shen}.   


The original problem can be generalized from different points of view. Therefore, to begin with, let us reflect for a moment at the genesis of the problem which actually combined the real applications referred to air or marine navigation and the corresponding theoretical research. First, the generalization of the initially formulated problem and its solution concerned dimension of the background Euclidean space. Zermelo's contribution was generalized shortly after by T. Levi-Civita \cite{levi} to an $n$-dimensional plane $\mathbb{R}^n$ and solved by means of variational calculus. Then, in particular, the problem was followed in detailed analysis with application of the Hamiltonian formalism by C. Carath\'{e}odory \cite{caratheodory}. Second, the generalization was connected to posing the problem in non-Euclidean scenario. In fact, it was focused on $\mathbb{S}^2$ for this was used as the model of earth. So in such a manner the solution has been extended to cover the case of flight on the surface of a sphere. The results by K. Arrow \cite{arrow} actually were applicable to any regular two dimensional surface since any such surface can be mapped conformally onto a plane. The obtained equations for the minimal path, motivated directly by real air navigation, also implied earlier results achieved by T. Levi-Civita. 

Setting as a reference point Zermelo's formulation of the problem the natural question arises, as yet another glance at the problem, whether a ship's own speed must be constantly maximal, i.e. $|u|_h=1$. This assumption we are going to drop considering the problem on Riemannian manifolds, however being in the case of a mild breeze which guarantees a full control of a navigating ship. Reviewing the bibliography in this scope one may find the paper by A. de Mira Fernandes \cite{mira} who accomodated shortly after Zermelo's contribution a varying magnitude velocity. Having added the extra degree of freedom the author allowed a time and space dependent velocity and solved the corresponding problem with the Euclidean background, namely in $\mathbb{R}^n$. Therefore, he has generalized the results of E. Zermelo and T. Levi-Civita for the Euclidean spaces to the case where the air speed is a preassigned function of position and time. De Mira Fernandes shows that the change in $|u|$ with time has no effect on the formula for the shortest time passage (time-optimal course) while that with space has the same effect as a corresponding change in wind \cite{arrow}. The contribution was referred in a modern approach by C. Herdeiro \cite{herdeiro} who discussed how, when both $W$ and $|u|$ are space but not time dependent, it can be recast in a purely geometric form as geodesics of a Randers geometry or as null geodesics in a stationary space-time. 

Finally, let us also mention the theoretical exposition by E. J. McShane in $\mathbb{R}^3$ \cite{mcshane}. 
The replacement of the sphere by the convex body (actually within the meaning of an indicatrix) was suggested by the fact that an airplane can travel faster down than up. McShane proved the existence of a solution of such a problem and considered the problem modified as to be a generalization of the Zermelo problem. That is the ship's velocity was required to be almost always as great as possible. Then he proved that  problem also was solvable. From such a point of view the problem seems to be closely related to another worth noticing contribution in Finsler geometry with application of a "slope-of-a-mountain" metric. There is an indicatrix of a man walking downhill and uphill on an inclined plane in $\mathbb{R}^3$ under acting the earth's gravity investigated \cite{matsumoto}. The latter contribution became the significant reference in later research on navigation problem in recent years, already strongly connected to Riemann-Finsler geometry. 

\

\section{Generalization of the problem and the solution}

Let a pair $ (M,h) $ be a Riemannian manifold where $h = h_{ij}dx^i\otimes dx^j$ is a Riemannian metric and the corresponding norm-squared of tangent vectors $y \in T_x M$ is denoted by $\left|y \right|_h^2 = h_{ij}y^iy^j = h(y, y)$. In the standard (classical) formulation we assume that a ship proceeds at constant speed relative to a perturbation, i.e. $|u|_h=1$, and the problem of time-efficient navigation is treated for the case of a mild wind $W$, namely $|W|_h < 1$ everywhere on $M$. The time-optimal paths on Riemannian manifold in unperturbed scenario are represented by the geodesics of the corresponding Riemannian metric. If we denote by $u$ the velocity of a ship in the absence of a wind then in the windy conditions a ship velocity will be given by the composed vector $v=u+W$. 

\ 



\subsection{Formula for modified Minkowski norm $\tilde{F}$}

We proceed being in line with the investigation on the original Zermelo navigation on Riemannian manifolds presented in \cite{colleen_shen} (ibid., pp. 381-384 for more details). After some modifications adapted to the new scenario and referring directly to the proof of Proposition 1.1 obtained therein we aim to point out differing consequences which come from both approaches. To reduce the clutter we also adopt the same notations if not otherwise stipulated. Given any Riemannian metric $h$ on a differentiable manifold $M$. In contrast to \cite{colleen_shen} we do not start with the fact $\left|u \right|_h =\sqrt{h(u, u)}= 1$. So far $W$ indicated the mild breeze with $|W|_h\in[0, 1)$. Now, first, we introduce the new assumption $|u|_h\in(|W|_h, 1]$. Hence, $|W|_h<|u|_h\leq1=u_{max}$. Both wind and ship's own speed are space-dependent, however they are not time-dependent. Into $h(u, u)$ we substitute $u = v-W$. Then $h(v, W ) = |v|_h |W|_h \cos \theta$ and $|u|_h^2=h(v-W, v-W)=|v-W|_h^2=|v|_h^2-2|v|_h|W|_h\cos\theta+|W|_h^2$, where \  $\theta\equiv\measuredangle\{v, W\}$. After using the abbreviation $ \tilde{\lambda} = |u|_h^2 - |W|_h^2$ we thus get  
$|v|_h^2 - 2 |v|_h|W|_h \cos \theta - \tilde{\lambda} = 0$. 
Since $|W|_h < |u|_h$ the resultant $v$ is always positive, hence $|v|_h > 0$. Solving the resulting quadratic equation and choosing the root that guarantees positivity yield 
\begin{equation}
|v|_h=|W|_h\cos\theta+\sqrt{|W|_h^2\cos^2\theta+|u|_h^2-|W|_h^2}, 
\end{equation}
which we can abbreviate as $\tilde{p} + \tilde{q}$. Since $\tilde{F}(v) = 1$, we see that
\begin{equation}
\tilde{F}(v) = \frac{|v|_h}{|v|_h}= |v|_h \frac {\tilde{q}-\tilde{p}} {\tilde{q}^2 - \tilde{p}^2} = \frac{ \sqrt{ \left[ h(W,v) \right]^2 + |v|_h^2 \tilde{\lambda}} - h(W,v)} {\tilde{\lambda}}.
\end{equation}
\noindent
Next, we need to deduce $\tilde{F} (y)$ for an arbitrary $y \in T M$. Every non-zero $y$ is expressible as a positive multiple $\tilde{c}$ of some $v$ with $\tilde{F}(v) = 1$. For $\tilde{c} > 0$, traversing $y=\tilde{c}v$ under the windy conditions should take $\tilde{c}$ units of time. This implies that $\tilde{F}$ is positively homogeneous,  $\tilde{F}(y)=c\tilde{F}(v)$. Using this
homogeneity and the formula derived for $\tilde{F}(v)$ the result is
\begin{equation}
\label{ran_mira}
\tilde{F}(y) = \frac{ \sqrt{ \left[ h(W,y) \right]^2 + |y|_h^2 (|u|_h^2-|W|_h^2)} } {|u|_h^2-|W|_h^2} - \frac{ h(W,y)} {|u|_h^2-|W|_h^2}.
\end{equation}
\noindent
The metric \eqref{ran_mira} is of Randers type. By hypothesis, $|W|_h < |u|_h$, hence $\tilde{\lambda} > 0$. We see that the formula for $\tilde{F}(y)$ is positive whenever $y \neq 0$ and $\tilde{F}(0) = 0$. The resulting Randers metric is composed of the modified new Riemannian metric and $1$-form in comparison to the analogous terms for the case of constant unit ship's speed through the water investigated in \cite{colleen_shen}. 


\subsection{Modified Riemannian metric and 1-form}

\noindent
The resulting Randers metric $\tilde{F}$ can also be presented in the form $\tilde{F} = \tilde{\alpha} + \tilde{\beta}$ as the sum of two components. Explicitly,
\begin{itemize}
\item  the first term is the norm of $y$ with respect to a new Riemannian metric $\tilde{a}$
\begin{equation}
\tilde{\alpha}(x,y)  = \sqrt{\tilde{a}_{ij}(x)y^iy^j}, \quad \text{ where } \quad \tilde{a}_{ij} = \frac{ h_{ij}}{\tilde{\lambda}} + \frac{W_i}{\tilde{\lambda}}\frac{W_j}{\tilde{\lambda}},
\label{alpha_mira}
\end{equation}

\item the second term is the value on $y$ of a differential 1-form
\begin{equation}
\tilde{\beta}(x,y)=\tilde{b}_i(x)y^i,\quad \text{where}\quad  \tilde{b}_i = -\frac{W_i}{\tilde{\lambda}},
\label{beta_mira}
\end{equation}
\end{itemize}
where $W_i=h_{ij}W^j$ and $\tilde{\lambda}=|u|_h^2- W^i W_i=|u|_h^2-h(W, W) 
=|u|_h^2-|W|_h^2.$

\begin{remark}
For $h(u, u)=1=const.$ the formulae \eqref{alpha_mira} and  \eqref{beta_mira} lead to the Riemannian metric $\alpha$ and 1-form $\beta$ according to \cite{colleen_shen}, respectively, in the case of the original   Zermelo navigation problem on Riemannian manifolds $(M, h)$, under a mild perturbation with $|W|_h<1$. 
\end{remark}
\noindent
With the presence of $W$ the time-optimal paths are the geodesics of the Finsler metric $\tilde{F}$. However, with the absence of a perturbation the geodesics of the background Riemannian metric $h$ are not necessarily the solutions to the problem as it is in the case of $h(u, u)=1$. The difference is made by the influence of the new factor $|u(x)|_h$ which now becomes a spatial function of $x$. We shall apply the metric $\tilde{F}$ which correlates to the background $h$-geodesics modified by $|u(x)|_h$.

\begin{remark}
A space-time viewpoint of the generalized problem was considered by C. Herdeiro (2009) who called it "Zermelo-Mira problem" and also, having followed \cite{colleen_shen}, included (independently of the author) the metric (\ref{ran_mira}); see the exposition in \cite{herdeiro}.    
\end{remark}


\subsection{Main findings}

One can ask if decreasing a ship's own speed $|u|_h$ under fixed wind field $W$ will cause the same effect on the time-optimal path as increasing the wind force with $|u|_h=1$ and maintaining the same relation $\frac{|W|_h}{|u|_h}$. Such an approach is in line with the comment in a space-time picture \cite{herdeiro} which states that since $0<|u|<1$ the decrease of the ship's velocity introduces a larger effective wind $\tilde{W}^i>W^i$. From this point of view the formulae for $\tilde{\alpha}$ and $\tilde{\beta}$ are exactly the same as given in \cite{colleen_shen} for the case $|u|=1$, however with $W^i$ replaced by a rescaled wind $\tilde{W}^i=\frac{1}{|u(x)|_h}W^i$, where $W\neq 0$. Furthermore, there is a triality between the null geodesic flow in $n$-dimensional stationary metrics, the geodesics flow in $n-1$-dimensional Randers geometries and the orbits which satisfy the navigation problem with the vector fields in $n-1$-dimensions which do not depend on time \cite{herdeiro}. What brings the efficient benefit is the fact that a problem which seems to be difficult to be solved in one of the backgrounds may occur simpler in another.  

Geometrically, in the standard formulation of the problem in each tangent space $T_xM$ the unit sphere of the metric $\tilde{F}$ is the $W$-translate of the Riemannian unit sphere. Now, in contrast to Proposition 1.1 in \cite{colleen_shen} the absence of a wind $W$ does not guarantee that the solutions to the problem are given by the background geodesics of $M$. A relevant example is presented further in Section 3. Thus, this implies that following the connections of greater $h$-length on $M$ can result in shorter travel time even when $W=0$, so not only in the presence of a perturbation. This differs from the solution to the original problem. The $h$-indicatrix is modified due to appearing the new variable $|u(x)|_h$. Thus, this fact leads to  
\begin{corollary}
The condition $W=0$ is not sufficient for the background $h$-geodesics to be the solutions to the generalized Zermelo navigation problem on $(M, h)$.  
\end{corollary}
From \eqref{ran_mira} with $W=0$ we get $\tilde{F}(y)=\frac{1}{|u(x)|_h}|y|_h=|y|_{\tilde{a}}$. Therefore, the solutions are given by $\tilde{a}$-geodesics. In particular, assuming additionally that $h(u(x), u(x))=const.$ we obtain the following corollary. 
\begin{corollary}
The solutions to the generalized Zermelo navigation problem on $(M, h)$ are $h$-geodesics up to normalization (scaling) if and only if $W=0$ and $h(u, u)=const.$ 
\end{corollary}
\noindent
It implies that the resulting Riemannian geodesics trace the same curves, {
however speeds differ. Let us recall that in the original Zermelo navigation in the case when $W$ vanishes $a=h$, where $\alpha^2(y)=a_{ij}y^iy^j$. Then the time-optimal solutions are represented by $h$-geodesics, that is $F(y):=\sqrt{h_{ij}y^iy^j}$. Let $|u(x)|_h$ be a smooth function of $x$ and by assumption $|u(x)|_h$ is positive. Now, in the absence of a wind the resulting Riemannian metric is conformal to the background metric $h$, $\tilde{a}\neq h$ for $h(u, u)\neq 1$. $\tilde{F}$ is Riemannian with $\tilde{a}_{ij}=\frac{1}{|u(x)|^2}h_{ij}$ and $\frac{1}{|u(x)|^2}$ plays the role of a conformal factor.  
The function $\tilde{F}$ is positive on the manifold $T M \setminus 0$, whose points are of the form $(x, y)$, with $0 \neq y \in T_x M$. Over each point $(x, y)$ of $T M \setminus 0 $ we designate the vector space $T_x M$ as a fiber, and name the resulting vector bundle $\pi^{\ast} T M$. A Finsler metric $\tilde{F}$ is said to be strongly convex if a canonical symmetric bilinear form $g_{ij} dx^i \otimes dx^j$ on the fibers of $\pi^{\ast} TM$ , with $g_{ij} = \frac 1 2 \left(\tilde{F}^2 \right)_{y^i y^j}$ is positive definite, in which case it defines an inner product on each fiber of $\pi^{\ast} T M$. The perturbation of Riemannian metrics $h$ by vector fields $W$ and the spatial function $|u(x)|_h$, with $|W|_h  < |u|_h\leq 1$, generates strongly convex Randers metrics. We proceed following the considerations in the proof of two implications in the case of $|u|=1=const.$ and adopting the theoretical assumptions mentioned therein to the new formulation of the problem with a variable in space $|u(x)|_h$. An inverse problem asks if every strongly convex Randers metric can be realized through the perturbation of some Riemannian metric $h$ by some vector field $W$ satisfying $h(W, W) < h(u, u)$ and taking into account $|u(x)|_h$. By this and making use of the above corollaries we thus get that $\tilde{F}$ is Riemannian, namely conformal to $h$ if and only if $W = 0$. It can be checked analogous to \cite{colleen_shen} by constructing $h$ and rescaled wind $\tilde{W}$ that perturbing the above $h$ by the stipulated $\tilde{W}$ gives back the Randers metric we started with. Summarizing, we obtain 
\begin{proposition}
A strongly convex Finsler metric $\tilde{F}$ is of Randers type if and only of it solves the generalized Zermelo navigation problem on a Riemannian manifold $(M, h)$, with a variable in space ship's own speed $|u(x)|_h\leq1$ and under the influence of a wind $W(x)$ which satisfies \linebreak $|W(x)|_h<|u(x)|_h $.  
Also, $\tilde{F}$ is Riemannian (conformal to $h$) if and only if $W=0$. In particular, $\tilde{F}$ is background Riemannian up to normalization (scaling) if and only if $W=0$ and $h(u, u)=const.$ 
\label{primero}
\end{proposition}
\noindent
Clearly, with the presence of a wind $W$ the Riemannian metric $h$ no longer gives the transit time along vectors. Furthermore, a glance at the new metric \eqref{ran_mira} leads to 
\begin{lemma}
\label{lemacik}
A transit time of a solution to the generalized Zermelo navigation problem on Riemannian manifolds, with arbitrary navigation data $(h, W)$ and any speed different from a constant unit speed, i.e. $|u(x)|_h\neq 1$, 
is greater than a transit time of the corresponding solution to the original Zermelo navigation problem. 
\end{lemma}
\begin{proof}
For any piecewise $C^{\infty}$ curve $\ell$ in $M$, the $\tilde{F}$- length of $\ell$ denoted by $\mathcal{L}_{\tilde{F}}(\ell)$ is equal to the time for which the object travels along it, i.e. $T= \int\limits_{0}^{T}\tilde{F}(\dot{\ell}(t))dt = \mathcal{L}_{\tilde{F}}(\ell).$ Let $\gamma$, $\tilde{\gamma}$ be $F$- and $\tilde{F}$-geodesic, respectively, where $F(y) = \frac{1}{\lambda}(\sqrt{ \left[ h(W,y) \right]^2 + |y|_h^2\lambda} - h(W,y))$ with $\lambda=1-|W|_h^2$ and $\tilde{F}(y) = \frac{1}{\tilde{\lambda}}( \sqrt{ \left[ h(W,y) \right]^2 + |y|_h^2\tilde{\lambda}}- h(W,y))$ with \  $\tilde{\lambda}=|u|_h^2-|W|_h^2$. 
For any nonzero $ y \in T_x M$ $F(y),\tilde{F}(y)>0$. Let us fix navigation data $(h, W)$, where $|W|_h<|u|_h\leq1$. By this, we have $\lambda\geq\tilde{\lambda}>0$. The equality of lengths $\mathcal{L}_{\tilde{F}}$ and $\mathcal{L}_{F}$ holds iff $|u|_h=1$, then $\tilde{F}(y):=F(y)$. Otherwise, $\tilde{F}(y) > F(y)$ for any scenario obtained from the triples $(h, W, |u(x)|_h)$ $\tilde{F}$ is a decreasing function with the minimum at $|u|_h=1$. 
Note that $\mathcal{L}_{\tilde{F}}(\tilde{\gamma})\geq\mathcal{L}_{F}(\tilde{\gamma})$ and $\mathcal{L}_{F}(\tilde{\gamma})\geq\mathcal{L}_{F}(\gamma)$ as the geodesic minimizes the length.  From transitivity we are thus led to the inequality $\mathcal{L}_{\tilde{F}}(\tilde{\gamma})\geq\mathcal{L}_{F}(\gamma)$.
\end{proof}

\ 

\noindent
Having introduced one more variable the bijection between Randers spaces and the triples $(h, W, |u|_h)$ or equivalently the pairs $(h,\tilde{W})$ is established. Note also that in Riemannian geometry two geodesics which pass through a common point in opposite directions necessarily trace the same curve. All reversible Finsler metrics have this property. However, this does not extend to nonreversible settings \cite{CR}. Thus, the differences in $\tilde{F}$-lengths indicate the differences in transit time on the Riemannian sea $(M, h)$.  



\section{Simulation with two-dimensional Randers metric} 

Let us assume that the initial Riemannian metric $h_{ij}$ to be perturbed is the standard Euclidean metric $\delta_{ij}$ on $\mathbb{R}^2$. In the scenario in which the new function $|u(x)|_h$ is taken into consideration the Randers metric $\eqref{ran_mira}$ coming from the navigaton data with the background Euclidean metric is expressed as follows 
\begin{equation}
\tilde{F}(x, y) = \frac{ \sqrt{( \delta_{ij}W^iy^j)^2 + |y|^2 (|u|^2-|W|^2)} } {|u|^2-|W|^2} - \frac{ \delta_{ij}W^iy^j} {|u|^2-|W|^2}.
\label{euklid}
\end{equation}
Since our focus is on dimension two, we denote the position coordinates $(x^1, x^2)$  by $(x, y)$, and expand arbitrary tangent vectors $y^1\frac{\partial}{\partial x^1}+y^2\frac{\partial}{\partial x^2}$ at  $(x^1, x^2)$ as $(x,y;u,v)$. We also express a ship's own speed $|u|_h$ as $|U|$. The formula \eqref{euklid} in two dimensional case after adopting the above notations, with an arbitrary mild vector field $W(x, y)=(W^1(x, y), W^2(x, y))$ yields
\begin{equation}
  \label{W1W2}
\tilde{F}(x,y; u,v) = \frac{\sqrt{(u^2+v^2)|U(x, y)|^2-(uW^2(x, y)-vW^1(x, y))^2}-uW^1(x, y)-vW^2(x, y)}{|U(x, y)|^2-|W(x, y)|^2}.
\end{equation}
Now, we apply the current $W$ determined by, for instance, the quartic plane curve given by 
\begin{equation}
W(x,y)=\hat{a} \left(\hat{b}-y^2\right)^2\frac{\partial}{\partial x} \qquad \text{with} \qquad  |W(x, y)|=|\hat{a}| \left(\hat{b}-y^2\right)^2 <|U(x, y)|\leq1.
\label{wind}
\end{equation}
\begin{figure}
    \centering
~\includegraphics[width=0.3\textwidth]{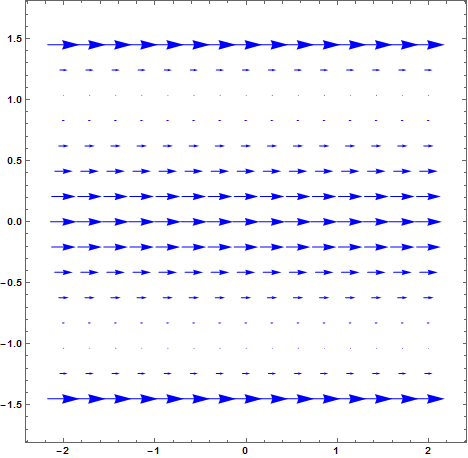}\qquad
~\includegraphics[width=0.13\textwidth]{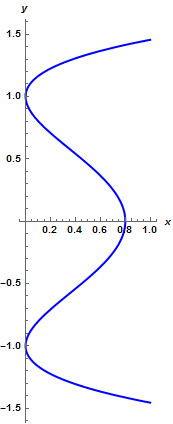}
~\includegraphics[width=0.37\textwidth]{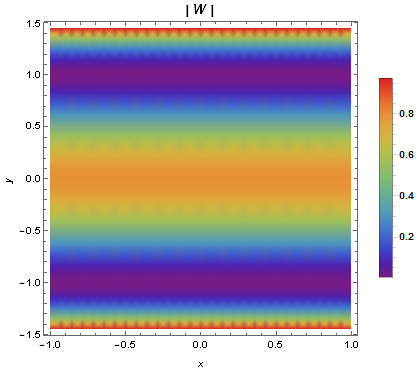}
\caption{The perturbation $W$ determined by the plane quartic curve \eqref{wind} with $\hat{a}=0.8$, $\hat{b}=1$.}
\label{rb_pole}
\end{figure}

\noindent
The current may have a rotational effect as well as a translational effect on a navigating ship and this effect depends on a ship's heading. Let $\hat{a}=0.8$, $\hat{b}=1$ and adapt the scenario to the horizontal flow of the perturbation. This is shown in Figure \ref{rb_pole}. Recall that to apply the above proposition the strong convexity condition must be fulfilled. Let the function determining the ship's own speed be $|U(x, y)|=\cos y$. In the Cartesian coordinate system which we apply the condition implies the domain  $|y|<y_0\approx 1.27$.  The graphs of the norms $|W(x, y)|$ and $|U(x, y)|$ satisfying $|W(x, y)|<|U(x, y)|\leq 1$ are compared in Figure \ref{u_vs_W}. The difference $|U(x, y)|-|W(x, y)|$ confirming required convexity is presented in blue dashed.  
For the perturbation \eqref{wind} the form of the resulting metric yields  
\begin{equation}
\tilde{F}(x,y,u,v)=\frac{\sqrt{\left(u^2+v^2\right) \cos ^2y-0.64 v^2 \left(y^2-1\right)^4}-0.8 u \left(y^2-1\right)^2}{\cos ^2y-0.64 \left(y^2-1\right)^4}.
\label{F_rb}
\end{equation}
\begin{figure}
        \centering
~\includegraphics[width=0.4\textwidth]{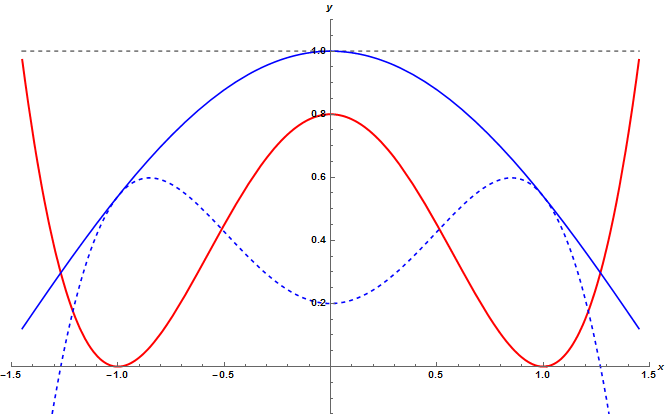}
        \caption{Comparing the norms $|W(x, y)|=0.8 \left(1-x^2\right)^2$ (red) and $|U(x,y)|=\cos x$ (blue) which satisfy $|W|<|u|\leq 1$ (in the example adapted to the horizontal flow); the difference $|U(x, y)|-|W(x, y)|$ is shown in blue dashed.}
\label{u_vs_W}
\end{figure}
\noindent
We use the partial derivatives of $\frac{1}{2}\tilde{F}^2$ to obtain the spray coefficients and then the final geodesic equations. Their solutions will determine the time-optimal paths. However, let us first recall that a spray on $M$ is a smooth vector field on $TM_0=TM\backslash \{0\}$ locally expressed in the standard form \cite{chern_shen} 
\begin{equation}
    \label{spray}
G=y^i\frac{\partial}{\partial x^i}-2G^i \frac{\partial}{\partial y^i},
\end{equation}
where $G^i=G^i(x, y)$ are the local functions on $TM_0$ satisfying $G^i(x, \hat{\lambda} y)=\hat{\lambda} ^2G^i(x, y)$, $\hat{\lambda}>0$. 
The spray is induced by $\tilde{F}$ and the spray coefficients $G^i$ of $G$ given by 
\begin{equation}
    \label{spray2}
G^i=\frac{1}{4}g^{il}\{ [\tilde{F}^2]_{x^ky^l}y^k-[\tilde{F}^2]_{}x^l\}
\end{equation}
are the spray coefficients of $\tilde{F}$.
\noindent
We consider the spray coefficients for general 2D Finsler metric. Let $L(x,y;u,v):=\frac{1}{2}\tilde{F}^2(x,y;u,v)$. Hence \cite{chern_shen},
\begin{equation}
G^1=\frac{L_{vv}(L_{xu}u+L_{yu}v-L_x)-L_{uv}(L_{xv}u+L_{yv}v-L_y)}{2(L_{uu}L_{vv}-L_{uv}L_{uv})},
\end{equation}
\begin{equation}
G^2=\frac{-L_{uv}(L_{xu}u+L_{yu}v-L_x)+L_{uu}(L_{xv}u+L_{yv}v-L_y)}{2(L_{uu}L_{vv}-L_{uv}L_{uv})}.
\end{equation}
Next, we apply the last two formulae to compute the spray coefficients of \eqref{F_rb}. For a standard local coordinate system $(x^i, y^i)$ in $TM_0$ the geodesic equation for Finsler metric is expressed in the general form $\dot{y}^i+2G^i(x, y)=0$. 
Hence, in the planar case the system of geodesic equations is reduced to
\begin{equation}
\label{geo}
\dot{x}+2G^1(x,y; \dot{x}, \dot{y})=0 \quad  \wedge \quad  \dot{y}+2G^2(x,y; \dot{x}, \dot{y})=0. 
\end{equation}
 
\noindent
The solution curves in the problem are found by working out the geodesics of $\tilde{F}$ given by \eqref{F_rb}. Analyses involving Randers spaces are generally difficult and finding solutions to the geodesic equations is not straightforward \cite{brody, chern_shen}. Also here, even in two-dimensional scenario due to the complexity of the spray coefficients' forms obtained we do not include the final time-optimal paths' equations. We use Mathematica 10.4 from Wolfram Research to generate the graphs and provide some numerical computations when the complete symbolic ones cannot be obtained. The numerical schemes can give useful information when studying the geometric properties of the solutions obtained as shown in the attached figures. 

In order to solve the final system of Randers geodesics in the example the form of the initial conditions including the optimal control $\varphi(t)$  become: for the point $x(0)=x_0\in \mathbb{R}$, $y(0)=y_0$, where we restrict $|y_0|\leq 1.25$, and for the tangent vector 
\begin{equation}
\label{ic_znp}
\dot{x}(0)=W^1(x_0, y_0)+|U(x_0, y_0)|\cos\varphi_0, \quad \dot{y}(0)=W^2(x_0, y_0)+|U(x_0, y_0)|\sin\varphi_0. 
\end{equation}
$\varphi(0)=\varphi_0\in [0, 2\pi)$ is the initial time-optimal angle measured counterclockwise which the relative velocity forms with $x$-axis. The dot indicates derivative with respect to $t$. The relations \eqref{ic_znp} can be derived by direct consideration of the equations of the planar motion including the representation of the vector components of $U$ and $W$. The results are presented in the figures below. The Randers geodesics in the case of the original Zermelo navigation (red) starting from a point in midstream, under the quartic curve perturbation \eqref{wind} are shown in Figure \ref{rb_7c_mira}. One can observe that there are the upstream focal points although there are none downstream. Much relevant information, in particular the effect of the variable in space $|U(x,y)|$, can be obtained directly from comparisons of the corresponding flows. We proceed setting for the families of $\tilde{F}$-geodesics in the example the following initial conditions: $x(0)=y(0)=0$, as in the case of $F$, and $\dot{x}(0)=0.8(1-y_0^2)^2+\cos\varphi_0$, $\dot{y}(0)=\sin\varphi_0$. We thus compare the Randers geodesics with constant $|U|=1$ (red) to the corresponding Randers geodesics with variable $|U|=\cos y$ (black) in the generalized problem, under the same quartic curve perturbation \eqref{wind}. In the demonstrated graphs we set the increments $\bigtriangleup\varphi_0$ which equal to $\frac{\pi}{9}$ or $\frac{\pi}{18}$. We see the influence of the new non-unit term $|U|$ on the final solution and, therefore, on the flow of the time-optimal paths.  
\begin{figure}[h]
        \centering
~\includegraphics[width=0.5\textwidth]{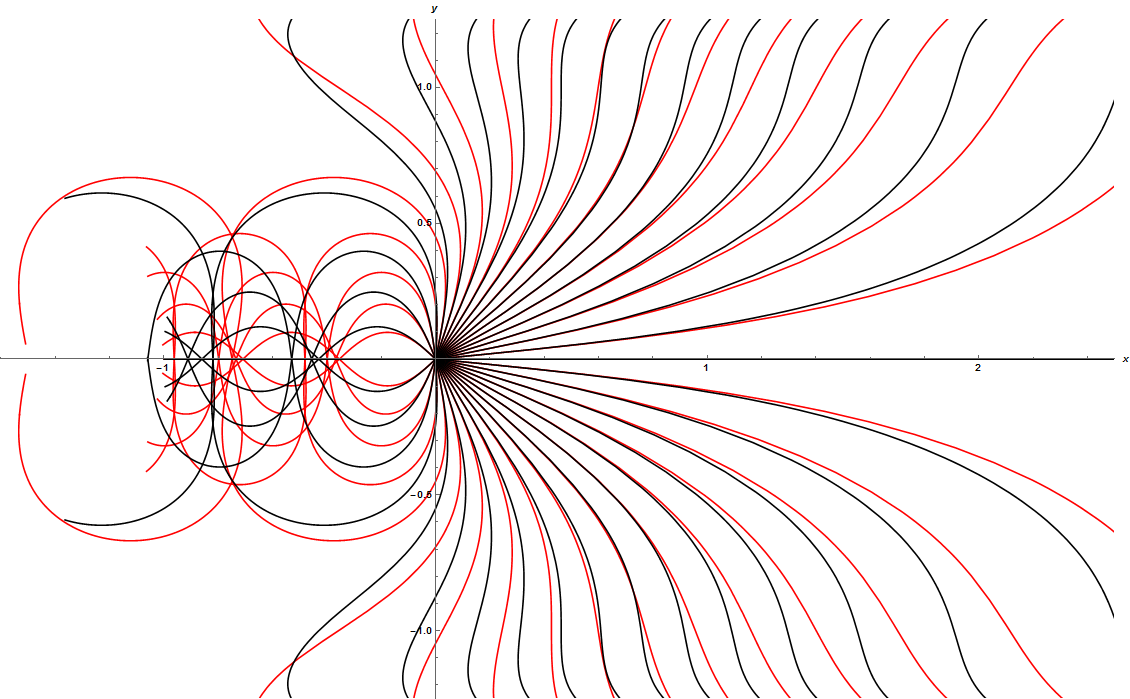}
~\includegraphics[width=0.45\textwidth]{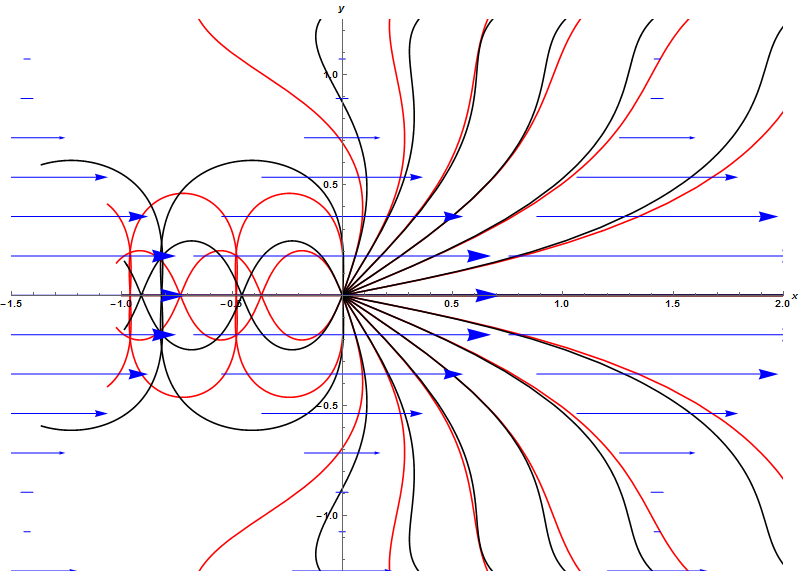}

        \caption{The Randers geodesics in the original Zermelo navigation, with $|U|=1$ (red) versus the geodesics in the generalized Zermelo navigation, with $|U(x, y)|=\cos y$ (black) starting from a point in midstream, under the same quartic curve perturbation \eqref{wind} marked by blue arrows, with the increments $\bigtriangleup\varphi_0=\frac{\pi}{18}$ (on the left) and $\bigtriangleup\varphi_0=\frac{\pi}{9}$ (on the right); $t=5$.} 
\label{rb_7c_mira}
\end{figure}[h]

Let us note that the presented example could also be investigated, for comparison, with the use of the original results (as the special case) by De Mira Fernandes \cite{mira} and partially by Zermelo \cite{zermelo, caratheodory} as we chose the planar Euclidean background. Moreover, their formulae for the time-optimal paths (headings) offer simpler computations. Nevertheless, we decided to show the Finslerian solution with the application of Randers metric as this approach enables to solve the analogous problems also in the non-Euclidean scenarios. 
\begin{figure}
        \centering
~\includegraphics[width=0.33\textwidth]{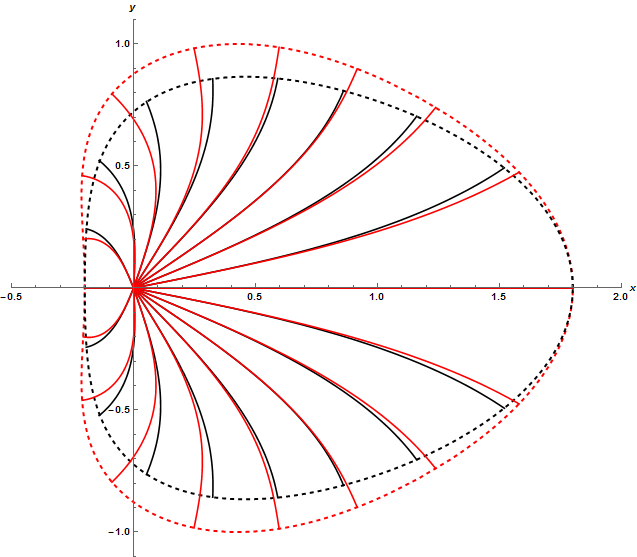} 
~\includegraphics[width=0.34\textwidth]{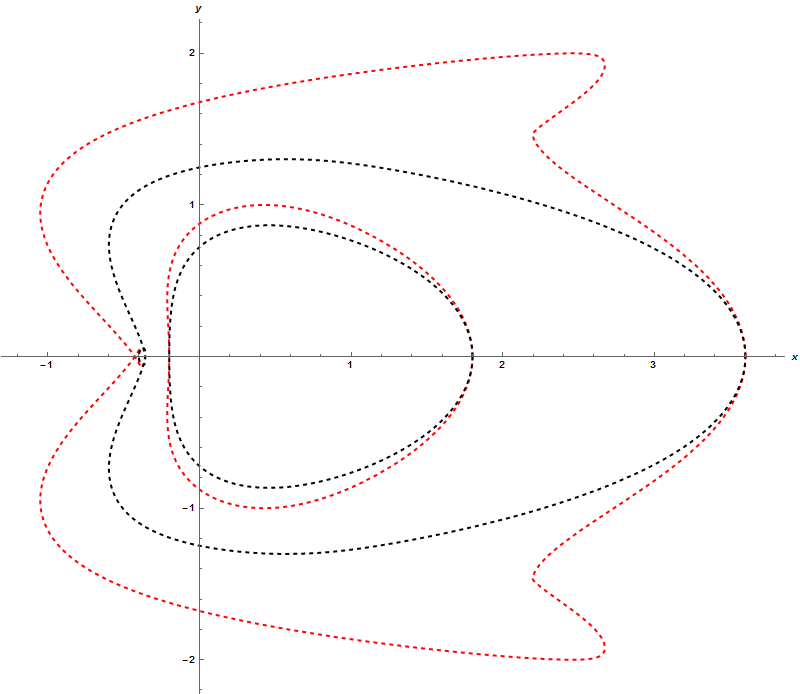} 
~\includegraphics[width=0.28\textwidth]{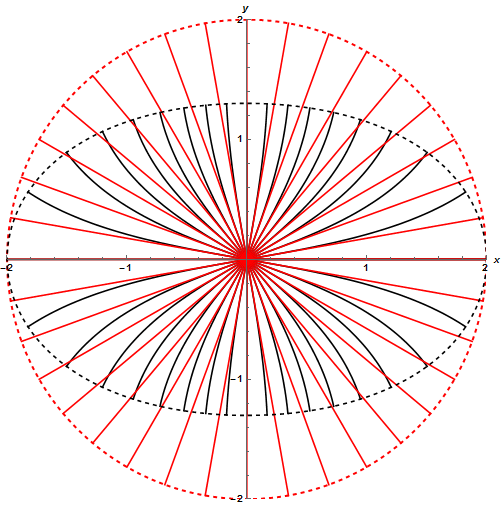}
        \caption{On the left comparing the unit indicatrices of $\tilde{F}$-geodesics (black dotted) and $F$-geodesics (red dotted), together with the corresponding geodesics (solid), $\bigtriangleup\varphi_0=\frac{\pi}{9}$. In the middle the comparison of $F$- and $\tilde{F}$-indicatrix with $t\in\{1, 2\}$. On the right $\tilde{F}$-indicatrix (black dotted) versus $F$-indicatrix (red dotted) in the absence of the perturbation ($W=0$), together with the corresponding geodesics (solid), $\bigtriangleup\varphi_0=\frac{\pi}{18}$, $t=2$.}
\label{speeds}
\end{figure}
\noindent
At first glance at the right-hand side graph in Figure \ref{rb_7c_mira} it could seem that a ship following $\tilde{F}$-geodesic, that is the time-optimal trajectory considering the variable $|U(x, y)|$, can reach given destination in some cases in shorter time than proceeding with a constant unit speed. For example, some of the black paths, so representing the passages at slower speed, reach the more distant points along the negative part of $x$-axis than the red ones in the same time ($t=5$). Such a deduction would be then in contradiction with Lemma \ref{lemacik}. This could be due to the fact that the set of the initial optimal angles $\varphi_0$ has been limited in the simulation. In fact, any given destination will be reached faster in the case of the original problem, however commencing with different $\varphi_0$. Both compared families are illustrated with the same values of the initial relative courses. In general, the variety of solutions depends here on the types of the space-dependent function $|U|$ and $W$. If we double the increments, i.e. $\bigtriangleup\varphi_0=\frac{\pi}{18}$ as it is presented in the left-hand side graph of Figure \ref{rb_7c_mira}, then we observe that another $F$-geodesic referring to different $\varphi_0$ will reach faster the desired crossing point of $F$- and $\tilde{F}$-geodesic. Indeed, the complete view is given by comparisons of the indicatrices as is shown in Figure \ref{speeds}. The left-hand side graph presents $\tilde{F}$-indicatrix (black dotted) together with $F$-indicatrix (red dotted) for $t=1$. Similarly, the corresponding indicatrices with $t\in\{1, 2\}$ are compared in the middle. Thus, we see that the former is included in the latter. The only crossing points are located on the horizontal axis. These refer to the cases when a ship sails exactly with or against the current $W$, i.e. $\varphi_0\in\{0, \pi\}$ and with a constant unit speed. In the proposed generalization of the problem this is the special case since $\tilde{F}$-geodesics become $F$-geodesics in the entire passage, i.e. $y=0\Rightarrow |U(x,y)|=\cos y=1$. The example is thus consistent with Lemma \ref{lemacik}. 

Finally, we also compare the solutions to the Zermelo navigation problem in the absence of the perturbation $W$ and the corresponding indicatrices with $t=2$. These are shown on the right in Figure \ref{speeds}. In the original version of the problem the solutions are obviously the straight lines (marked by the red solid rays coming from the origin) so the indicatrices are the concentric circles. Now, taking into consideration $|U(x, y)|$ which is not constant in space the solutions (black solid) are not $h$-geodesics any more. Instead, they are the geodesics of the new Riemannian metric $\tilde{a}$ conformal to $h$, where $|(u, v)|_{\tilde{a}}=\frac{1}{\cos y}\sqrt{u^2+v^2}$. In the example they form the ellipse-shaped indicatrices. Following the background $h$-geodesics connecting two given points on $M$ in the absence of a wind gives longer travel time than corresponding $\tilde{a}$-geodesics which are of greater $h$-length. This fact also makes a difference in comparison to the original approach to the navigation problem.  



\section{Conclusions}

In the paper we investigated the generalization of the Zermelo navigation in a purely geometric form admitting a space dependence of a ship's own speed. Proposition \ref{primero} establishes the direct relation between the Randers geodesics and the time-optimal paths as the solutions to the navigation problem introducing additional degree of freedom. Having one more variable $|u(x)|_h$ the bijection is established between Randers spaces represented by pairs  $(\tilde{\alpha}, \tilde{\beta})$ and triples $(h, W, |u|_h)$ or pairs $(h,\tilde{W})$ in the original approach, where $W^i$ is replaced by a modified wind $\tilde{W}^i=\frac{1}{|u(x)|_h}W^i$ with $|W|\neq 0$. 
In this way it generalizes the earlier theorem, i.e. Proposition 1.1 in \cite{colleen_shen} which became the point of reference and the motivation for our research. Also, the presented approach increases the variety of the scenarios and the corresponding solutions influenced by the new spatial  function $|u(x)|_h$. In particular, in contrast to the original version of the problem the absence of a wind ($W=0$) is not the sufficient condition for the background $h$-geodesics to be the solutions to the generalized Zermelo navigation problem on $(M, h)$. This holds now up to scaling if additionally $h(u(x), u(x))=const.$ Such a case occurs, in particular, in the original Zermelo navigation on Riemannian manifolds which now becomes the special case of the generalized problem. The solutions represented by $h$-geodesics are then "slower", that is with greater transit time in comparison to the standard formulation since $h(u, u)\neq1$. Furthermore, according to Lemma \ref{lemacik} the same also holds for non-Riemannian solutions, that is under acting vector field $W$ as for any $|u(x)|_h$ different from a constant unit speed $\tilde{F}$-length $>$ $F$-length. The above fact also implies that following the connections of greater $h$-length on $M$ can result in shorter time even when $W=0$, so not only in the presence of a perturbation. This differs from the original approach to the navigation problem. Admitting a space dependence of a ship's own speed also brings us closer to the real applications, for instance in air or marine navigation (cf. \cite{kopi4}). Although we expanded the scenarios and their solution there is still something to yield. More advanced generalization would also admit both: the time dependence of acting perturbation $W(t,x )$ and a ship's own speed $|u(t, x)|_h$ as well as the strong winds, i.e. $h(W, W)\geq1$. Note that from the early beginning the first contributions on the problem with the Euclidean backgrounds which we referred to in the introduction coincide with the above issues (cf. \cite{mira, zermelo, zermelo2, mcshane}). 

To complete, let us mention and gather some other purely geometric contributions expanding the investigations on the Zermelo navigation such as with the use of Kropina metric \cite{kropina} in Finsler geometry or in reference to the Lorentzian metric with the general relativity background, e.g. \cite{sanchez} as well as the variational solution with the application of the Euler-Lagrange equations \cite{palacek}, followed by the author in \cite{kopi}. We see that the initially formulated problem can be indeed generalized from different points of view. Therefore, the Zermelo navigation can be researched as the problem itself and the technique applied to solve other open geometric problems. The former seems to be closer to optimal control and the latter to differential geometry. In fact, the problem combines fruitfully the notions of both areas and physics showing the efficient equivalences (dualities). A task which seems to be difficult to be solved in one field may occur simpler in another making use of Zermelo's problem. Having in mind the above generalizations we conclude that it might be interesting to consider the navigation problem with the non-Riemannian background metrics. 


\ 

\noindent
\textbf{Acknowledgement} \quad The research was supported by a grant from the Polish National Science Center under research project number 2013/09/N/ST10/02537.  


\bibliographystyle{plain}
\bibliography{pk}

\end{document}